\documentclass[12pt]{article}
\usepackage[toc,page]{appendix}
\usepackage{adjustbox}
\usepackage{multirow}
\usepackage{bigstrut}
\usepackage{graphics}
\usepackage{tikz}
\usetikzlibrary{shapes,backgrounds,calc}
\usepackage{latexsym}
\usepackage{multicol}
\usepackage{tikz}
\usepackage{tikz-cd}
\usepackage{mathtools}
\usepackage{float}
\usepackage{array}
\usepackage{longtable}
\usepackage{stackengine}
\usepackage{amsmath,amssymb,latexsym,amsfonts,amscd,pgf,comment,enumerate,ragged2e}
\usepackage{amsthm}
\usepackage{graphicx}
\usepackage{float}
\usepackage{setspace}
\usepackage[pdftex]{hyperref}
\makeindex
\usepackage{multicol}
\theoremstyle{plain}
\newtheorem{theorem}[equation]{Theorem}
\newtheorem*{FundClaim*}{Fundamental Claim}
\newtheorem{lemma}[equation]{Lemma}
\newtheorem{corollary}[equation]{Corollary}
\newtheorem{proposition}[equation]{Proposition}
\theoremstyle{definition}

\newtheorem{example}[equation]{Example}

\begin{document}
\bigskip

\begin{center}
	\Large{\bf \large A Note on Strongly $\pi$-Regular Elements}
	\bigskip
	
	\large{Dimple Rani Goyal}
	
	\bigskip
\end{center}

\medskip
\begin{abstract}
\begin{small}
In this article, we prove that in a PI-ring (or polynomial identity ring) $S$, for an element $A \in \mathbb{M}_m(S)$ if $A^n= A^{n+1}X$ for some $n \in \mathbb{N}$ and $X \in \mathbb{M}_m(S)$, then there exists an element $Y\in \mathbb{M}_m(S)$ such that $A^n = YA^{n+1}$. As a consequence, we show that this property also holds in matrix rings over commutative rings, thereby confirming a recent conjecture proposed by Călugăreanu and Pop. Moreover, we present another independent proofs of this conjecture, highlighting different structural approaches and techniques. \\[4mm]
{\em Keywords:} Strongly $\pi$-regular elements, Strongly regular elements, Exchange rings, Prime factor rings, Dedekind-finite rings, PI-rings.\\[3mm]
Mathematics Subject Classification 2020: 16E50, 16U99
\end{small}
\end{abstract}
\section{Introduction}
In 1948, Kaplansky and Arens \cite{ArensKap} defined a ring $R$ to be \emph{strongly regular} if, for every element $a \in R$, there exists an element $y \in R$ such that $a = a^{2}y$. The term \emph{strongly regular} emphasized that such rings satisfy a stronger condition than ordinary regular or biregular rings, and because the defining relation appeared symmetric. In \cite{ArensKap}, they also studied rings in which every element $a$ satisfies the more general condition $a^{n} = a^{n+1}x$ for a fixed $n$. Two years later, Kaplansky \cite{KP} proved that this more general definition is indeed symmetric. Although the defining condition $a = a^{2}x$ (or its left analogue $a = ya^{2}$)is symmetric globally, it is not element-wise symmetric in general. Kaplansky \cite{KP} initiate the study of rings in which every element $a$ satisfies $a^{n} = a^{n+1}x$, where both $x$ and $n \in \mathbb{N}$ are independent of $a$. Subsequently, in 1953, Azumaya \cite{Az} formally introduced the terminology \emph{right regular element} and \emph{right $\pi$-regular element} if $a$ satisfies the conditions $a = a^{2}x$ and $a^{n} = a^{n+1}x$ respectively. Much later, in 2014, Dittmer et al. \cite{Dit} revised the terminology, renaming right $\pi$-regular elements as \emph{right strongly $\pi$-regular} for greater clarity. Motivated by this convention, in the present paper we adopt the term \emph{right strongly regular} in place of \emph{right regular}. The same notation was used in \cite{CalPop}.\\

\par Recall that an element $a$ of ring $R$ is called {\em right strongly $\pi$-regular} if $a^n = a^{n+1}x$ for some $n \in \mathbb{N}$ and $x \in R$. If $n=1$, then $a$ is said to be {\em right strongly regular}. Left strongly $\pi$-regular and left strongly regular elements are defined similarly. An element which is both left and right strongly $\pi$-regular (left and right strongly regular) is called {\em strongly $\pi$-regular} (respectively, {\em strongly regular}). A ring in which every element is right strongly $\pi$-regular is called right strongly $\pi$-regular. The left versions of these rings are defined similarly. In 1976, Dischinger \cite{Dischinger} proved that a ring is left strongly $\pi$-regular if and only if it is right strongly $\pi$-regular. Nowadays, these rings are simply called strongly $\pi$-regular.\\

\par It is easy to see that Dischinger's result is not true locally. For example, if $R$ is not Dedekind-finite and there exist $a,\,b\in R$ such that $ab = 1$ but $ba \neq 1$, then $aR = a^2R = R$ but $Ra \neq Ra^2$.\\ {
 \par In \cite{CalPop}, C\u{a}lug\u{a}reanu and Pop mentioned that with the help of computer, they proved that Dischinger's result holds locally for $2 \times 2$ matrices over the ring $\mathbb{Z}_n$ for $n=4,9,16$ and $\mathbb{Z}$. Also, they proposed a conjecture \cite[Conjecture 3.7]{CalPop} that in $ \mathbb{M}_2(S)$, where $S$ is a commutative ring, every left strongly regular element is right strongly regular. Following this, a natural question arise that if every left strongly regular element is strongly regular in ring $R$, does same holds in $\mathbb{M}_n(R)$. In the present work, we aim to resolve this question not only for commutative rings but also for a broader class of rings. This question is also interesting in view of Morita theory. Recall that a ring theoretic property $P$ is Morita invariant if and only if whenever a ring $R$ satisfies $P$ so do $eRe$ (for any full idempotent $e \in R$) and $\mathbb{M}_n(R)$ (for any integer $n \geq 2$).\\
\par Recall that a ring $R$ is PI-ring (or polynomial identity ring) if it satisfies a monic polynomial $g \in \mathbb{Z}<x_1,..., x_k>$ where $x_i$ are noncommuting indeterminates. PI rings are very large and interesting class of rings. For more details on these rings (see \cite{Rowen}). It is well known that the homomorphic image of PI ring is again PI. \\
\par In this short note, firstly we prove that for any $A \in  \mathbb{M}_m(S)$, where $S$ is a PI-ring if $A^n\mathbb{M}_m(S) = A^{n+1}\mathbb{M}_m(S)$ for some positive integer $n$ then  $\mathbb{M}_m(S)A^n = \mathbb{M}_m(S)A^{n+1}$. As a corollary, we prove that the conjecture of C\u{a}lug\u{a}reanu and Pop  \cite[Conjecture 3.7]{CalPop} is true. Lastly, we provide an example (essentially due to Shepherdson) in a ring $R$ in which Dischinger's result holds locally but fails to hold in $M_{2}(R)$.
\\

\section{Results}
Firstly, We list some results required from the literature and their consequences. The following result was proved by Drazin \cite[Theorem 4]{Dr}(see also \cite[Corollary 1.2]{GK2}. The next two results also hold for rings without unity.

\begin{lemma}\label{N1}
If $a$ is a strongly $\pi$-regular element in a ring $R$ and $a^n = a^{n+1}x$, then for $w = a^nx^{n+1}$, we have  $aw = wa$, $a^n = a^{n+1}w$.
\end{lemma}
 Dittmer et al. proved the following result in \cite[Lemma 4.4]{Dit}. We modify this result for rings without unity in the following way and provides a more simpler proof using Theorem \ref{N1}.  Note that, in the next result, it is not necessary for the subring $S$ to have a unity.
 
\begin{proposition}\label{N2}
Let $S$ be a subring of ring $R$ and $a\in S$ be such that $a^n = a^{n+1}x$ for some positive integer $n$ and $x \in S$. If $a$ is strongly $\pi$-regular in $R$, then there exists an element $y \in S$ such that $a^n = ya^{n+1}$. In particular, if Dischinger's result holds locally in $R$, then it also holds in $eRe$, where $e$ is an idempotent of $R$.
\end{proposition}
\begin{proof}
As $a^n = a^{n+1}x$ and $x \in S$. Since $a$ is strongly $\pi$-regular in $R$, so by Lemma \ref{N1}, for $w = a^nx^{n+1}$, we have $wa = aw$ and $a^n = a^{n+1}w$. Since $w = a^nx^{n+1} \in S$, $a$ is strongly $\pi$-regular in $S$.\qed
\end{proof}
\vspace{2mm}

In \cite[Theorem 2.1]{FS} Fisher and Snider proved that a ring $R$ is strongly $\pi$-regular if it so modulo every prime ideal. Their proof also gives the following element-wise result. 

\begin{theorem}\label{N3}
Let $a \in R$ be such that $a+P$ is strongly $\pi$-regular in $R/P$ for every prime ideal $P$ of $R$, then $a$ is strongly $\pi$-regular in $R$.
\end{theorem}
Using Proposition \ref{N2} and Theorem \ref{N3}, we have the following result. 
\begin{theorem}\label{N3.1}
Let $R$ be a ring whose every prime factor ring embeds in a ring in which every right strongly $\pi$-regular element is strongly $\pi$-regular. Then every right strongly $\pi$-regular element in $R$ is also strongly $\pi$-regular.
\end{theorem}

Recall that a ring $R$ is {\em strongly $2$-primal} if every prime ideal of $R$ is completely prime. Hence, by Theorem \ref{N3.1}and as in every domain Dischinger 's result holds,  we can prove the following result easily,
\begin{corollary}
Dischinger's result holds locally in Strongly $2$-primal rings.    
\end{corollary}
In \cite[Corollary 3.5]{Dit}, Dittmer et al.\ proved the following result.
\begin{lemma}\label{Dfer}
 If $R$ is a Dedekind-finite exchange ring, then Dischinger's result holds locally.
\end{lemma}
In view of above result, we have the following corollary of Theorem \ref{N3.1}.
\begin{corollary}
Let $R$ be a ring whose every prime factor ring embeds in a Dedekind-finite exchange ring. Then Dischinger's result holds locally in $R$.
\end{corollary}
Azumaya \cite[Lemma 3]{Az} proved that if $a \in R$ is strongly $\pi$-regular with $a^n = a^{n+1}x$ and $a^m = ya^{m+1}$ for some $x,\,y \in R$ and $n,\,m \in \mathbb{N}$, then $a^m = a^{m+1}x$ and $a^n = ya^{n+1}$. The following results follows quickly from this and we will frequently use this result.\\ 

The following result can be deduced from \cite{posner} for PI-rings:

\begin{theorem}\label{N4}
Any prime PI-ring is a subring of a matrix ring over a division ring.
\end{theorem}
Using the above result, we have the following result.
\begin{theorem}\label{PI}
Let $R$ be a PI-ring and $a \in R$ be such that $a^n=a^{n+1}x$ for some $n \in \mathbb{N}$ and $x \in R$, then $a^n=ya^{n+1}$ for some $y \in R$.
\end{theorem}
\begin{proof}
Let $P$ be any prime ideal of $R$. As $R/P$ is prime ring satisfying some polynomial identity, thus by Theorem \ref{N4}, $R/P$ is subring of $\mathbb{M}_n(D)$ for some division ring $D$. Since $\mathbb{M}_n(D)$ is artinian, so $a+P$ is strongly $\pi$-regular in $\mathbb{M}_n(D)$. By Proposition \ref{N2}, $a+P$ is strongly $\pi$-regular in $R/P$ for any prime ideal $P$. Therefore,  $a$ is strongly $\pi$-regular in $R$ by Theorem \ref{N3}. Lastly, by Azumaya's result \cite[Lemma 3]{Az}, we have the required result.
\end{proof}

\medskip

It is already known that matrix ring over a PI-ring is again PI (\cite[Exercise 1.8]{Rowen}), Thus, we have the following immediate result. \\
\begin{corollary}
Let $A\in \mathbb{M}_m(S)$, where $S$ is a PI-ring and $m \in \mathbb{N}$. Then for any positive integer $n$, $A^n = A^{n+1}X$ implies $A^n= YA^{n+1}$ for some $Y \in \mathbb{M}_m(S)$.
\end{corollary}

The polynomial identity $$s_n(x_1, x_2, \dots, x_n) = \sum_{\sigma \in S_n} \operatorname{sgn}(\sigma) \, x_{\sigma(1)} x_{\sigma(2)} \cdots x_{\sigma(n)}$$ where $S_n$ is the symmetric group of order $n$ is of great importance in the theory of PI-rings and is known as the $nth$ standard identity. \\

Amitsur and Levitzki \cite[Theorem 1]{AmitsurLevitzki} proved that the matrix ring $\mathbb{M}_n(R)$ over a commutative ring $R$ is PI-ring and satisfies the polynomial identity $s_{2n}$. Thus, for $m=2$, the following result validates \cite[Conjecture 3.7]{CalPop}.

\begin{corollary}\label{comm}
 Let $A\in \mathbb{M}_m(S)$, where $S$ is a commutative ring and $m \in \mathbb{N}$. Then for any positive integer $n$, $A^n\mathbb{M}_m(S) = A^{n+1}\mathbb{M}_m(S)$ implies $\mathbb{M}_m(S)A^n= \mathbb{M}_m(S)A^{n+1}$.	
\end{corollary}

The above corollary can also be proved by alternative methods available in the literature. In fact, we provide three additional proofs, the second of which is similar in spirit to the proof of the Theorem \ref{PI}.\\

The following result is classically known. For its history and a proof see \cite[Corollary 1.2]{grover2022matrices}. Let $f(x) \in \mathbb{M}_n(S)[x]$, where $S$ is a commutative ring. We can regard $f(x)$ to be an element of $\mathbb{M}_n(S[x])$ and so we can define $\det(f(x))$ naturally. 
\begin{proposition}\label{N5}
Let $S$ be a commutative ring. If $A \in \mathbb{M}_n(S)$ satisfies $f(x) \in  \mathbb{M}_n(S)[x]$, then $A$ satisfies $\det(f(x))I$, where $I$ is the identity matrix of $\mathbb{M}_n(S)$.
\end{proposition}
Recall that a ring $R$ is called {\em clean} if every element is a sum of a unit and an idempotent. It is well known that every clean ring is an exchange ring. Burgess and Raphael proved the following result in \cite[Proposition 2.4]{burgess}.

\begin{proposition}\label{N7}
Every commutative ring embeds in a commutative clean ring.
\end{proposition}

Now, we provide alternative proofs of Corollary \ref{comm}. In view of Azumaya's result \cite[Lemma 3]{Az}, it is enough to show that $A$ is strongly $\pi$-regular in $\mathbb{M}_m(S)$. 

\medskip
\noindent {\bf First Proof.} Let $A^n = A^{n+1}B$ for some $B \in  \mathbb{M}_m(S)$. As $A$ satisfies a polynomial $x^nI - x^{n+1}B \in \mathbb{M}_m(S)[x]$, so by Proposition \ref{N5},  $A$ also satisfies $\det(x^nI - x^{n+1}B)I$ which is a polynomial of the type $Ix^{nm} + s_1I x^{nm+1} +\ldots + s_{m}I x^{nm+m} \in \mathbb{M}_m(S)[x]$. This implies that $A^{nm} = A^{nm+1}f(A)$ for some polynomial $f(x)\in \mathbb{M}_m(S)[x]$ implying that $A$ is strongly $\pi$-regular in $\mathbb{M}_m(S)$. \qed

\bigskip
\noindent {\bf Second Proof.} By Proposition \ref{N7}, $S$ embeds in a commutative exchange ring, say $T$. So $\mathbb{M}_m(S)$ embeds in a Dedekind-finite exchange ring $\mathbb{M}_m(T)$. By Lemma \ref{Dfer}, $A$ is strongly $\pi$-regular in $\mathbb{M}_m(T)$ and hence $A$ is strongly $\pi$-regular in $\mathbb{M}_m(S)$ using Proposition \ref{N2} .\qed

\bigskip

In \cite{Shep}, Shepherdson constructed the following ring $R$ to answer a question raised by van der Waerden: Does there exists an element $A \in \mathbb{M}_2(R)$ that is both right invertible and left zero divisor in $\mathbb{M}_2(R)$? This example is also useful in demonstrating that, even in rings where all left (and right) strongly regular elements are strongly regular, the matrix rings over such rings do not necessarily inherit this property.
\begin{example}
	
Let $X = \{a, b, c,d,w,x,y,z\}$ be a set of indeterminates. Consider \[R = \mathbb{Q} \langle a, b, c,d,w,x,y,z\rangle / \langle aw + by - 1, \; ax + bz, \; cw + dy, \; cx + dz - 1 \rangle.\]

For \[
A = \begin{bmatrix}
	a & b \\
	c & d
\end{bmatrix} and \quad
B = \begin{bmatrix}
	w & x \\
	y & z
\end{bmatrix}\]
we have, $AB=I$, where $I$ is the $2 \times 2$ identity matrix.
In other words,
since $R$ is a domain then it is easy to see that every left (or right) strongly regular element is strongly regular. The construction is in such a way that clearly  $\mathbb{M}_2(R)$ is not even Dedekind-finite.  
\end{example}

		
		\bibliographystyle{amsplain}
		\bibliography{Dimplebib}
		
	\end{document}